\documentclass[10pt]{amsart}
\usepackage{amsmath, amscd, amssymb}
\usepackage[frame,cmtip,arrow,matrix,line,graph,curve]{xy}
\usepackage{graphpap, color,bbm, verbatim}
\usepackage[mathscr]{eucal}

\newcommand{\CC}{\mathbb{C}}

\newcommand{\PP}{\mathbb{P}}

\newcommand{\ZZ}{\mathbb{Z}}

\newcommand{\bP}{\mathbf{P}}

\newcommand{\bM}{\mathbf{M}}

\newcommand{\cal}{\mathcal}

\def\cN{{\cal N}}
\def\cO{{\cal O}}

\def\cU{{\cal U}}

\newcommand{\rH}{\mathrm{H} }

\def\PP{\mathbb{P}}
\def\CC{\mathbb{C}}
\def\lr{\rightarrow}

\newtheorem{prop}{Proposition}[section]
\newtheorem{theo}[prop]{Theorem}
\newtheorem{lemm}[prop]{Lemma}
\newtheorem{coro}[prop]{Corollary}
\newtheorem{rema}[prop]{Remark}
\newtheorem{defi}[prop]{Definition}

\newtheorem{defiprop}[prop]{Definition-Proposition}

\DeclareMathOperator{\Ext}{Ext} 
\DeclareMathOperator{\Hom}{Hom} 
\newcommand{\ses}[3]{0\lr{#1}\lr{#2}\lr{#3}\lr 0}

\DeclareMathOperator{\Pic}{Pic}
\def\deg{\mathrm{deg}}

\def\Hom{\mathrm{Hom}}
\def\Ker{\mathrm{Ker}}
\def\Coker{\mathrm{Coker}}

\def\rH{\mathrm{H}}

\def\bM{\mathbf{M}}

\def\bP{\mathbf{P}}

\title[Stable maps in the vector bundles space]{Stable maps of genus zero in the space of stable vector bundles on a curve}

\author{Kiryong Chung*}\thanks{* Corresponding author.}
\address{Department of Mathematics Education, Kyungpook National University, 80 Daehakro, Bukgu, Daegu 41566, Korea}
\email{krchung@knu.ac.kr}

\author{SangHyeon Lee}
\address{Department of Mathematical Sciences, Seoul National University, GwanAkRo 1, Seoul 08826, Korea}
\email{tlrehrdl@snu.ac.kr}

\keywords{Stable bundles; Modification of vector bundles; Rational curves}
\subjclass[2010]{14E05, 14H60, 14D22}
\thanks{KC was partially supported by NRF grant 2016R1D1A1B03930421.}

\begin{document}
\begin{abstract}
Let $X$ be a smooth projective curve with genus $g\geq3$.
Let $\mathcal{N}$ be the moduli space of stable rank two vector bundles on $X$ with a fixed determinant $\mathcal{O}_X(-x)$ for $x\in X$.
In this paper, as a generalization of Kiem and Castravet's works, we study the stable maps in $\mathcal{N}$ with genus $0$ and degree $3$.
Let $P$ be a natural closed subvariety of $\mathcal{N}$ which parametrizes stable vector bundles with a fixed subbundle $L^{-1}(-x)$ for a line bundle $L$ on $X$. 
We describe the stable map space $\mathbf{M}_0(P,3)$. It turns out that  the space $\mathbf{M}_0(P,3)$ consists of two irreducible components. One of them parameterizes smooth rational cubic curves and the other parameterizes the union of line and smooth conics.
\end{abstract}
\maketitle
\section{Introduction}
\subsection{Stable maps in the space of stable vector bundles}
Let $i:Y\subset \PP^r$ be a smooth projective variety with $i^*\cO_{\PP^r}(1)=\cO_Y(1)$. Let $\bM_0(Y,d)$ be the moduli space parameterizing stable maps $[f:C\lr Y]$ with genus $g(C)=0$ and degree $\mathrm{deg}(f):=d$. It is well-known that the moduli space $\bM_0(Y,d)$ is compact and its geometry has been studied in various contexts: enumerative geometry (\cite{FP97, KP01}) and birational geometry (\cite{CC10}). 
In this paper, as a continuation of \cite[\S 3]{Kie07}, we study the moduli space of stable maps of degree $3$ when the target $Y=\cN$ is the moduli space of stable vector bundles of rank two over a smooth curve $X$. Here $\cN$ parameterizes the Gieseker-Mumford stable vector bundles $E$ of rank $2$ with fixed determinant $\mathrm{det}(E)=\cO_X(-x)$ on a smooth projective curve $X$ (\S 2.1). It is well-known that the ample theta divisor $\Theta$ generates the Picard group $\mathrm{Pic}(\cN)$ (\cite{BV99}). For the projective embedding $i:\cN\subset \PP^N$ provided by the divisor $\Theta$ on $\cN$, one can consider the moduli space $\bM_0(\cN,d)$ of stable maps with degree $d$.
As the degree of the maps become larger, the moduli space $\bM_0(\cN,d)$ may have many different irreducible components. We study the problem from lower degree cases. Let us summarize the well-known results for $d\leq2$.
\begin{itemize}
\item When $d=1$, the moduli space is isomorphic to a projective bundle over the Picard group $\mathrm{Pic}^0(X)$ (\cite{Kil98, Mun99}).
\item When $d=2$, the moduli space $\bM_0(\cN,2)$ consists of two irreducible components: Hecke curves and the rational curves of extension types where they transversally intersect. Furthermore, the moduli space $\bM_0(\cN,2)$ is related to the \emph{Hilbert scheme} of conics in $\cN$ by using the birational morphisms with some geometric meaningful centers (\cite{Kie07}).
\end{itemize}
When $d=3$, it is quite natural to start with finding the possible irreducible components of the moduli space $\bM_0(\cN,3)$. In this direction, in \cite{Cas04, Kie07}, the authors found the two components parametrizing the degree 3 map $f : \PP^1 \to \cN$.  For $L\in \text{Pic}^k(X)$, let $\PP_L:=\PP\Ext^1(L,L^{-1}(-x))$ be the space of non-split extensions \[\ses{L^{-1}(-x)}{E}{L}.\]
By the functoriality of the moduli space $\cN$, we have a natural rational map
$$
\Psi_L:\PP_L \dashrightarrow \cN.
$$
Two components parametrize one of the following types of degree 3 maps $f:\PP^1 \to \cN$. (\cite[Lemma 4.10]{Cas04}, \cite[Proposition 3.9]{Kie07}).
\begin{itemize}
\item[i)] When $k=0$, $f$ is a composition of degree 3 maps $\PP^1 \to \PP\Ext^1(L,L^{-1}(-x))\cong \PP^{g-1}$ with the map $\Psi_L$. In this case $\Psi_L$ is a linear embedding (i.e., $\mathrm{deg}\Psi_L=1$). Also it is well known that every $(g-1)$-dimensional linear space arise in this fashion. Hence one of the components of $\bM_0(\cN,3)$ is isomorphic to the relative stable map space $\bM_0(\PP\cU/\text{Pic}^0(X),3)$ where $\cU$ is the universal rank $g$ bundle over $\text{Pic}^0(X)$. Thus, the method in \cite{CK11, CHK12} comparing the various compactifications of rational curves can be applied.
\item[ii)] When $k=1$, $f$ is a composition of the degree one map $\PP^1 \to \PP\Ext^1(L,L^{-1}(-x))\cong \PP^{g+1}$ (i.e., line) with the map $\Psi_L$. In this case, $\mathrm{deg}\Psi_L=3$. But the base locus of the map $\Psi_L$  may not be empty. Moreover, the rational map $\Psi_L$ may not be injective. 
\end{itemize}
In fact, the base locus $\mathrm{Bs}(\Psi_L)$ is isomorphic to $X$ from \cite{Tha94, Ber92}. Also, the resolution of the undefined locus of the rational map $\Psi_L$ is the first blow-up $\pi: M_1=bl_X\PP^{g+1}\lr \PP^{g+1}=\Ext^1(L, L^{-1}(-x))$ in \cite{Tha94} (Proposition \ref{mainprop}). Here the space $M_1$ is the moduli space of pairs on $X$ (For definition, see \cite{Tha94}). 
If $L$ is \emph{non-trisecant} (Definition \ref{nontri}), the regular morphism $M_1\lr \cN$ is a closed embedding and thus the study of stable maps in $\cN$ boils down to the study of stable maps in $M_1$.
\subsection{Main results}
Let $L$ be a non-trisecant line bundle on $X$ (Definition \ref{nontri}). Let $ \pi^*[\mathrm{line}]:=\beta\in \rH_2(\text{bl}_X \PP^{g+1})$ be the pull-back of the line class along the blow-up map $\pi: bl_X\PP^{g+1}\lr \PP^{g+1}$. Then we have a closed embedding $$\bM_0(bl_X\PP^{g+1}, \beta) \hookrightarrow \bM_0(\cN,3)$$ between the moduli spaces of stable maps (Proposition \ref{mainprop}). There may be many irreducible components in $\bM_0(bl_X\PP^{g+1}, \beta)$ (cf. \cite{KP01, Gat96, KLO07}). 
One of the obvious components of the moduli space $\bM_0(bl_X\PP^{g+1}, \beta)$ is the closure of the locus of the lines $l\subset \PP^{g+1}$ such that $l\cap \mathrm{Bs}(\Psi_L)=\emptyset$ (i.e., an open subset of $Gr(2,g+2)$). Let us denote it by $\Gamma_1$. There is another irreducible component in the moduli space $\bM_0(bl_X\PP^{g+1}, \beta)$ which arises from the blowing-up $\pi$. The component consists of stable maps with reducible domain whose image is the union of a line and conic such that the conic lies in the exceptional locus of the blow-up. Let us denote it by $\Gamma_2$. One of our main goals of this paper is to prove that these are all of the components. That is,
\begin{theo}[Theorem \ref{comp}]
Under the above notations, the moduli space $\bM_0(bl_X\PP^{g+1}, \beta)$ consists of the two irreducible components $\overline{\Gamma}_1$ and $\overline{\Gamma}_2$ of dimension $2g$. Furthermore, the intersection part $\overline{\Gamma}_1\cap \overline{\Gamma}_2$ is $2g-1$-dimensional irreducible space.
\end{theo}
In fact, we can describe the moduli points in $\overline{\Gamma}_1\cap \overline{\Gamma}_2$ with the deformation theory of maps (\S 3).

Regarding enumerative geometry, the moduli space of stable maps in the blown-up space of the projective space has been intensively studied in \cite{KLO07} and \cite{Gat96}.
\subsection{Notations} Throughout this article, we use the following notation.
\begin{itemize}
\item $X$: projective curve of genus $g\geq 3$.
\item $x$: a fixed point of $X$.
\item $V_L:=\Ext^1(L,L^{-1}(-x))$, $V_L^{s}$ is the stable part of $\Ext^1(L,L^{-1}(-x))$.
\item If there is no need to emphasize a line bundle $L\in \Pic^1(X)$(respectively, $L \in \Pic^0(X)$) and extension groups, we sometimes abbreviate $\PP V_L$ by $\PP^{g+1}$ (respectively, $\PP^{g-1}$). And we sometimes abbreviate stable locus of $\PP V_L:=\PP V_L^s$ by $\PP^s$, when $L \in \Pic^1(X)$ in the same situation.  
\end{itemize}

\medskip
\textbf{Acknowledgement.}
We would like to thank W. Lee and A. Iliev for valuable discussion and comments. The second named author is grateful to his thesis advisor, Young-Hoon Kiem.


\section{Review of the resolution of unstable bundles}

\subsection{Stable bundles and stable maps}
Let $\mu (E):=\mathrm{deg}(E)/\mathrm{rank}(E)$ be the slope of the vector bundle $E$ on $X$. A vector bundle $E$ is called \emph{stable} if $$\mu(F)<\mu(E)$$ for all non-zero, proper, subbundles $F\subset E$.
Let $\cN$ be the moduli space of stable rank $2$ vector bundles $E$ on $X$ with fixed determinant $\text{det}(E)\cong \cO_X(-x)$.  
Since $(\mathrm{rank}(E), \mathrm{deg}(E))=1$, from geometric invariant theory (\cite{Tha96}), it is well-known that the moduli space $\cN$ is a smooth projective variety.

On the other hand, for a smooth projective variety $Y$, let $\beta \in \rH_{2}(Y, \ZZ)$. Let $C$ be a projective connected reduced curve of genus $g(C)=0$. A map $f: C \to Y$ is called \emph{stable} if
\begin{itemize}
\item $C$ has at worst nodal singularities;
\item $|\mathrm{Aut}(f)| < \infty$.
\end{itemize}
Let $\bM_{0}(Y, \beta)$  be the moduli space of stable maps with $f_{*}[C] = \beta$. Then $\bM_{0}(Y, \beta)$ is a projective scheme (\cite[Theorem 1]{FP97}).

\subsection{Some remarks about the rational map $\Psi:\PP^{g+1}\dashrightarrow \cN$}
\begin{defiprop}(\cite{Kie07})\label{modify}
Let $E$ be a rank two vector bundle on $X$ and $p\in X$.
Let $E^{v_p}$ be the kernel of the surjective map
\begin{equation}\label{eqmod}
0\longrightarrow E^{v_p}\longrightarrow E \stackrel{v_p}{\longrightarrow} \CC_p\longrightarrow 0.
\end{equation}
Then $E^{v_p}$ is a vector bundle. Let $E^{v_p}$ be \emph{an elementary modification} of the bundle $E$ at $p$.
\end{defiprop}
Since $\Hom(E,\CC_p)=\Hom(E|_p,\CC_p)=\CC^2$ and $\Ker (v_p) = \Ker (\lambda \cdot v_p)$ for all $\lambda \in \CC^*$, we assume that $v_p \in (\CC^2\setminus \{0\})/\CC^* \cong \PP^1$.
Let $E=\zeta\oplus \zeta'$ for line bundles $\zeta$ and $\zeta'$ on $X$. If $v_p \in \CC^*=\PP^1\setminus \{[1:0], [0:1]\}$, one can easily see that the modified bundles $E^{v_p}$ are isomorphic to each other. Hence, from now on, let us denote
\begin{defi}\label{smodify}
Let us define by
\[
(\zeta\oplus \zeta')^p:= (\zeta \oplus \zeta')^{v_p}
\]
for all $v_p\in \CC^*$.
\end{defi}
Let $K_X$ be the canonical line bundle on $X$.
\begin{lemm}(cf. \cite[(3.4)]{Tha94} and \cite[\S 3]{Ber92})\label{setmap}
Let $$f:X \to \PP^{g+1}=\PP \Ext^1(L,L^{-1}(-x)),\quad p\mapsto (L\oplus L^{-1}(p-x))^p$$ be the map provided by the elementray modification. Then $f$ is equal to the map given by the complete linear system $$i=|L^2(x)\otimes K_X|:X\subset \PP^{g+1}.$$
\end{lemm}
\begin{proof}
In \cite[(3.4)]{Tha94}, it was proved that the map $i$ is equivalent to the map $g:X \to \PP H^1(\Lambda^{-1})=M_0$ ($\Lambda=L^2(x)$) where $M_0$ is the moduli space of stable pairs on $X$. 
In fact, the map $g$ is defined by  by $\PP W_1 \to \PP H^1(L^{-2}(-x))$, where $W_1$ is a line bundle on $X$ and $g(p)=\PP H^0(L^{-2}(-x)|_p)\in \PP H^1(L^{-2}(-x))$ (the final paragraph of \cite[329p]{Tha94}). This is the projectivization of the first map $\xi$
in the following exact sequence:
\begin{equation}\label{eq5}0 \to \Ext^1(L|_p,L^{-1}(-x)) \stackrel{\xi}{\to} \Ext^1(L,L^{-1}(-x)) \stackrel{\gamma}{\to} \Ext^1(L(-p),L^{-1}(-x)) \to 0,
\end{equation}
where \eqref{eq5} is obtained by taking the functor $\Hom(-,L^{-1}(-x))$ to the short exact sequence $\ses{L(-p)}{L}{L|_p}$.
Since $\Ext^1(L|_p,L^{-1}(-x))=\CC$, to prove that $g(p)=f(p)$,  it is enough to show that $\gamma(f(p))=L(-p)\oplus L^{-1}(-x)$. That is, 
\begin{equation}\label{eq15}
(L\oplus (L^{-1}(p-x))^p\oplus_L L(-p) \cong L(-p)\oplus L^{-1}(-x),
\end{equation}
where the left hand side is defined by the pull-back:
\[\xymatrix{ 
0 \ar[r] & L^{-1}(-x) \ar[r] \ar@{=}[d] & (L\oplus L^{-1}(p-x))^p\oplus_L L(-p)   \ar[r] \ar[d] & \ar[r] L(-p) \ar@{^{(}->}[d] \ar[r] & 0 \\
0 \ar[r] & L^{-1}(-x) \ar[r] & (L\oplus L^{-1}(p-x))^p \ar[r] & L \ar[r] & 0.
}\]
One can check the isomorphism in \eqref{eq15} locally as follows. For every open set $U\subset X$,
\begin{align*}
&((L\oplus L^{-1}(p-x))^p\oplus_L L(-p))(U)\\
&=\{((s_1,s_2),s_3)|(s_1,s_2)\in ((L\oplus L^{-1}(p-x))^p(U) ,s_3\in L(-p)(U),s_1=s_3\}\\
&\cong \{(s_1,s_2)\in((L\oplus L^{-1}(p-x))^p(U)|s_1\in L(-p)\}\\
&\cong \{(s_1,s_2)\in L(U)\oplus L^{-1}(p-x)(U))|as_1(p)+bs_2(p)=0,s_1(p)=0 \} \\
&\cong \{(s_1, s_2)\in L(U)\oplus L^{-1}(p-x))(U)|s_1(p)=s_2(p)=0\}\\
&\cong \{(s_1,s_2)\in (L(-p) \oplus L^{-1}(-x))(U)\},
\end{align*}
where the third isomorphism comes from the definition of the elementray modification and the fourth comes from the choice of $v_p=[a:b]\in\CC^*$ for $ab\neq 0$.
\end{proof}
\begin{rema}
$\text{deg}(X)=2g+1$.
\end{rema}
\begin{prop}(\cite{Tha94, Ber92})\label{mainprop}
Let $$\Psi_L:\PP^{g+1}\dashrightarrow \cN$$ be the rational map defined by the extension of $L$ by $L^{-1}(-x)$. Then
\begin{enumerate}
\item the undefined locus of the rational map $\Psi_L$ is isomorphic to $X$ (Lemma \ref{setmap}) and the blowing-up of $\PP^{g+1}$ along $X$ extends to a regular morphism $$\widetilde{\Psi}_L: \text{bl}_X \PP^{g+1} (:=\bP)\longrightarrow \cN.$$
\item The exceptional divisor $\PP^{g-1}$ restricted to a fiber of the blowing-up $\pi$ is exactly the degree $0$ extension type and thus each restricted exceptional divisor is a linearly embedding into $\cN$ by $\widetilde{\Psi}_L$.
\item If $\Psi_L$ is injective and $H^0(L^2(x))=0$, then the morphism $\widetilde{\Psi}_L$ is a closed embedding.
\[
\xymatrix{
\bP\ar[dr]^{\widetilde{\Psi}_L}\ar[d]_{\pi}&\\
\PP^{g+1}\ar@{-->}[r]^{\Psi_L}&\cN.
}
\]
\end{enumerate}
\end{prop}
\begin{proof}
Let us follow the notation of \cite{Tha94} by letting $\Lambda=L^2(x)$ and thus $\ses{\cO_X}{E}{L^2(x)}$. Also, $\bP\cong M_1$ where the latter space parameterizes the pairs $(s,E)$ such that $E$ is stable and $s\subset \rH^0(E)$ (\cite{Tha94}).
Part (1) comes from Lemma \ref{setmap} and \cite[(2.1)]{Tha94}. Part (2) comes from item (2) of \cite[Theorem 1]{Ber92}.
In part (3), the injectiveness of $\widetilde{\Psi}$ comes from \cite[(3.20)]{Tha94} because $\rH^0(E)=\CC$. Note that our extended map $\widetilde{\Psi}_L$ is just the forgetful map $(s,E)\mapsto E$ by forgetting the section $s\subset \rH^0(E)$. Hence the tangential map $\widetilde{\Psi}_{*} :T_{[(s,E)]} \bP \lr T_{[E]} \cN$ fits into
the exact sequence (\cite[(2.1)]{Tha94})
\[
0\lr \Ext^0(E,E)\lr H^0(E)\lr T_{[(s,E)]} \bP \stackrel{\widetilde{\Psi}_{*} }{\longrightarrow} T_{[E]} \cN.
\]
Since $\Ext^0(E,E)=\CC$ and $H^0(E)=\CC$, we conclude that the map $\widetilde{\Psi}$ is embedding. 
\end{proof}
The assumptions of the part (3) of above proposition can be satisfied by the following choice of the line bundles $L$ (cf. Lemma \ref{ninj}).
\begin{defi}\label{nontri}
The line bundle $L$ is \emph{non-trisecant} if $\rH^0(L^2(x))=0$.
\end{defi}

If $L$ satisfies the non-trisecant condition then we have the following property.
\begin{coro}(\cite[Lemma 5.1]{Hom83})\label{trisec}
\begin{itemize}\item[(a)] The smooth curve $X \subset \PP^{g+1}$ (embedded by $|L^2(x)\otimes K_X|$) has a trisecant line if and only if

\begin{equation}\label{eq1}
L^2(x)\cong\cal{O}_X(p+q+r) (\mathrm{ equivalently, } \rH^0(L^2(x))\neq0)
\end{equation}
for some $p,q,r\in X$.
In this case, there is the trisecant line such that its intersection with $X$ is $p+q+r$ (if $p=q$, then the line is tangential to $X$ at $p$ and if $p=q=r$ then the line is tangential to $X$ at $p$ and meets $p$ three times).
\item[(b)] If the curve $X$ is neither trigonal nor hyperelliptic, the points $p,q,r$ satisfying equation \eqref{eq1} are uniquely defined and thus there exists a unique trisecant line.
\end{itemize}
\end{coro}
\begin{proof}
For part (a), let $l$ be a trisecant line. Let $l\cap X=\{p,q,r\}$(the three points $p, q, r$ may be equal to each other), then $l=\overline{pq}$. 
Since $l$ is trisecant, by the similar method in Proposition \ref{line}, the line $l$ is the projectivization of the kernel of a map $j_1$ which fits into the diagram
\begin{eqnarray}
\Ext^0(L,L^{-1}(-x)) \to \Ext^0(L(-p-q-r),L^{-1}(-x)) \to \Ext^1(L|_{p+q+r},L^{-1}(-x)) \nonumber\\ \stackrel{j_1}{\to} \Ext^1(L,L^{-1}(-x)) \to \Ext^1(L(-p-q-r),L^{-1}(-x)) \to 0.
\end{eqnarray}
Because the first term is clearly zero and dimension of the kernel is $2$, dimension of
$\Ext^0(L(-p-q-r),L^{-1}(-x))\cong \rH^0(L^{-2}(-x)\otimes\cO_X(p+q+r))$  is $1$. Therefore $L^{-2}(-x)\otimes\cO_X(p+q+r)\cong \cO_X$. So we conclude that $L^2(x)\cong \cO_X(p+q+r)$.

Conversely, let us assume that the line bundle $L$ satisfies equation \eqref{eq1}. By taking the functor $\Hom(L,-)$ in the short exact sequence $\ses{L^{-1}(-x)}{L^{-1}(p+q+r-x)}{L^{-1}(p+q+r-x)|_{p+q+r}}$, we obtain
\begin{eqnarray}\label{seq2}
\Ext^0(L,L^{-1}(p+q+r-x)) \to \Ext^0(L,L^{-1}(p+q+r-x)|_{p+q+r}) \to \Ext^1(L,L^{-1}(-x)) \nonumber\\
\stackrel{j_2}{\rightarrow} \Ext^1(L,L^{-1}(p+q+r-x)) \to  \Ext^1(L,L^{-1}(p+q+r-x)|_{p+q+r})=0.
\end{eqnarray}

Since final term is clearly zero,
$\dim(\Ext^1(L,L^{-1}(-x)))=g+2$ and $\dim\Ext^1(L,L^{-1}(p+q+r-x))=\dim\Ext^1(L,L)=g$ by \eqref{eq1}. Thus $\dim\ker \pi=2$. Also, since $j_2$ is equal to the composition of $j_1$ and an isomorphism $\Ext^1(L(-p-q-r),L^{-1}(-x))\cong \Ext^1(L,L^{-1}(p+q+r-x))$, $\ker j_2=\ker j_1$. Therefore $\ker j_2$ is exactly the same as the affine cone of the linear subspace $l:=\langle p,q,r\rangle$ in $\PP^{g+1}$. Hence the linear subspace $l$ is a line in $\PP^{g+1}$ such that $X\cap l=p+q+r$.

For part (b), suppose that there are three points $\{p',q',r'\}(\neq\{p,q,r\})$ such that $L^2(x)\cong \cal{O}_X(p'+q'+r')$. Then $\cal{O}_X(p+q+r-p'-q'-r')\cong \cal{O}_X$. This means $X$ is hyperelliptic or trigonal.
\end{proof}

\subsection{Geometry of lines in $\PP^{g+1}$ meeting X}
In this subsection, we give a description of line the $\overline{pq}:=\langle f(p),f(q) \rangle$ passing through $p,q\in X\stackrel{f}{\hookrightarrow}\PP^{g+1}=\PP V_L$ for $L\in \Pic^1(X)$.
If $p=q$, then $\overline{pp}$ denotes the projectivized tangent line of $X$ at $f(p)$.
Recall that the image $f(t)$ (Lemma \ref{setmap}) fits into the following short exact sequences:
$$
0 \to L^{-1}(-x) \to f(t) = (L \oplus L^{-1}(t-x))^t \to L \to 0.
$$
\begin{prop}\label{line}
Under the above notation, $M:=L\oplus L^{-1}(p+q-x)$. Then the line $\overline{pq}\setminus \{p,q\}$ missing the points $p$ and $q$ is
parameterized by the \emph{doubly} modified bundles $(M^{v_p})^{v_q}(=(M^{v_q})^{v_p})$ which fit into the exact sequence:
\begin{equation*}
0\longrightarrow (M^{v_p})^{v_q}\longrightarrow M\stackrel{(v_p \oplus  v_q)}{\longrightarrow} \CC_{p}\oplus\CC_{q}\longrightarrow 0.
\end{equation*}
Here $v_p \in \CC^* \subset \PP(\rH^0(M|_p)^*)= \PP^1$ and $v_q \in \CC^* \subset \PP(\rH^0(M|_q)^*)= \PP^1$.
Hence $\overline{pq}\setminus \{p,q\}=\PP V_L^s\cap\PP V_{L^{-1}(p+q-x)}$.
\end{prop}
\begin{proof}
By some diagram chasing, one can easily show that the doubly modified bundle is exactly the kernel of the map $v_p\oplus v_q$.

Let us describe the line $\overline{pq}$.
By taking the functor $\Hom (-, L^{-1}(x))$ into the short exact sequence $\ses{L(-p-q)}{L}{L|_{p+q}}$,
we obtain
\begin{align*}
0=\Ext^0(L(-p-q),L^{-1}(-x)) \to \Ext^1(L|_{p+q},L^{-1}(-x)) \stackrel{i}{\to} \Ext^1(L,L^{-1}(-x)) \\ \stackrel{j}{\to} \Ext^1(L(-p-q),L^{-1}(-x)) \stackrel{\varphi}{\cong} \Ext^1(L,L^{-1}(p+q-x)) \to 0
\end{align*}
where $\varphi$ is obtained from the twisting by $\cO(p+q)$ and the first equality holds for degree reasons.
We claim that
$$
\PP\Ext^1(L|_{p+q},L^{-1}(-x))=\overline{pq}.
$$
To show this, at first, we prove that the line $\PP\Ext^1(L|_{p+q},L^{-1}(-x))\subset \PP^{g+1}$ is parameterized by the elements defined by the double modifications. 
 Let $E\in \Ext^1(L|_{p+q},L^{-1}(-x))$. The image $E$ by the map $i$ fits into the short exact sequence:
\[ 0 \to L^{-1}(-x) \to i(E) \to L \to 0. \] Since $\varphi(j(i(E)))= L^{-1}(p+q-x)\oplus L$, we have the following push-out diagram.
\begin{equation}\label{eq23}
\xymatrix{0 \ar[r] & L^{-1}(-x) \ar[r] \ar@{_{(}->}[d] & i(E) \ar[r] \ar[d]^a & L \ar[r] \ar@{=}[d] & 0 \\
0 \ar[r] & L^{-1}(p+q-x) \ar[r] & L^{-1}(p+q-x)\oplus L \ar[r] & L \ar[r] & 0.}
\end{equation}
By some diagram chasing, we have the short exact sequence:
\begin{equation}\label{eq21}
0\lr i(E) \stackrel{a}{\longrightarrow} L^{-1}(p+q-x)\oplus L \stackrel{v_p \oplus v_q}{\longrightarrow} \CC_{p+q} \lr 0
\end{equation}
such that $p_1 \circ a$ is surjective where the map $p_1:L\oplus L^{-1}(p+q-x) \lr L$ is the projection into the first factor. Also it is clear that $p_1 \circ a$ is surjective if and only if $v_p \neq [1:0]$ and $v_q \neq [1:0]$.

Conversely, let us assume that the bundle $E$ fits into the sequence \eqref{eq21} such that $p_1 \circ a$ is surjective. 
Then one can easily check that the $E$ fits into the diagram \eqref{eq23} by the snake lemma.

In summary, the extensions of the form \eqref{eq21} such that $v_p \neq [1:0]$ and $v_q \neq [1:0]$
represent elements in $\Ext^1(L|_{p+q},L^{-1}(-x))$. 
Secondly, one can easily show that  the modified bundle $f(p)$ (resp. $f(q)$) fits into the diagram in \eqref{eq21} if $v_q=[0:1]$ and $v_p \in \CC^{*}$ (resp. $v_p=[0:1]$ and $v_q\in \CC^{*}$). 
Hence if $p\neq q$, the two distinct points $f(p)$ and $f(q)$ lie on the line $\PP\Ext^1(L|_{p+q},L^{-1}(-x))$ which is equal to $\overline{pq}$. 

Since the extension $\Ext^1(L|_{2p},L^{-1}(-x))$ can be obtained as the limit of the extension $\Ext^1(L|_{p+q},L^{-1}(-x) )$ by letting $p \to q$,  the same result holds for the case $p=q$.
\end{proof}
Now we describe the set-theoretic intersections of the vector bundles provided by the extensions:
\begin{enumerate}
\item $\PP V_{\zeta}^s \cap \PP V_{\eta}$ for $\zeta \in \Pic^1(X)$ and $\eta \in \Pic^0(X)$;
\item $\PP V_{\zeta} \cap \PP V_{\eta}$ for $\zeta,\eta \in \Pic^0(X)$.
\end{enumerate}
We remark that all of these intersections are considered in the moduli space $\cN$. Case (2) has been already studied in \cite[6.19]{NR78}. They meet at a single point cleanly unless the intersection is empty. Hence we focus on case (1).
\begin{prop}\label{intersec}
Let $\zeta\in \Pic^1(X)$ and $\eta\in \Pic^0(X)$. Then
the intersection $\PP V_{\zeta}^s \cap \PP V_{\eta}$ in $\cN$  is one of the following:
\begin{itemize}
\item[i)] If $\zeta\otimes \eta \cong \cO(p+q-x)$,  then $\PP V_{\zeta}^s \cap \PP V_{\eta}$ is equal to the image of $\overline{pq}\setminus \{p,q\}$ in $\cN$. 
\item[ii)] Otherwise, $\PP V_{\zeta}^s \cap \PP V_{\eta} = \phi$
\end{itemize}
\end{prop}
\begin{proof}
Consider an element $E$ in $\PP V_{\zeta}^{\mathrm{s}}\cap\PP V_{\eta}$. Then,
\[ \xymatrix@=10pt{0 \ar[r] & \zeta^{-1}(-x) \ar[r]^(0.65)a & E \ar[r]^b \ar[d]^= & \zeta \ar[r] & 0 \\
0 \ar[r] & \eta^{-1}(-x) \ar[r]^(0.65)c & E \ar[r]^d & \eta \ar[r] & 0.}\]
If $d\circ a$=0, $d$ factors through $\zeta$. But since $\deg(\zeta)=1>\deg(\eta)=0$, $d=0$ which leads to contradiction. So, $d\circ a$ is injective. Since $\deg(\eta)=0$ and $\deg(\zeta^{-1}(-x))=-2$, $\Coker(d \circ a))= \CC_{p+q}$ for $p,q\in X$.  Therefore $\zeta^{-1}(-x)\cong\eta(-p-q)$. From the commutative diagram
\begin{equation}\label{comdg1}
\xymatrix{& &0  &0&\\
0\ar[r] &L^{-1}(-x)\ar@{=}[d] \ar[r] & \eta \ar[r]^{r} \ar[u]  &\CC_{p+q} \ar[u]\ar[r] &0\\
 0 \ar[r] & \zeta^{-1}(-x) \ar[r]^{a} & E \ar[r]^{b} \ar[u]^d & \zeta \ar[r] \ar[u]_{s} & 0\\
&  & \eta^{-1}(-x) \ar@{=}[r] \ar[u]^{c}  & \eta^{-1}(-x) \ar[u]_{b\circ c} & \\
& & 0\ar[u] & 0,\ar[u] &}
\end{equation}
we can see that $\Coker(d \circ a) \cong \Coker(b \circ c)$.
From this,
\[ \xymatrix@=15pt{0 \ar[r] & \zeta^{-1}(-x) \ar[r]^(0.65)a & E \ar[r]^b \ar[d]^= & \zeta \ar[r] & 0 \\
0 \ar[r] & \zeta(-p-q) \ar[r]^(0.65)c & E \ar[r]^(0.25)d & \zeta^{-1}(p+q-x) \ar[r] & 0.} \]
Hence we obtain a morphism $$b\oplus d :E \lr  \zeta\oplus\zeta^{-1}(p+q-x).$$

Let us consider the diagram
\begin{equation}\label{comdg2}
\xymatrix{ 0 \ar[r] & \CC_{p+q} \ar[r]^(.4){\Delta} & \CC_{p+q}\oplus \CC_{p+q} \ar[r]^(.6){h} & \CC_{p+q} \ar[r] & 0 \\
0 \ar[r] & E \ar[r]^(.3){b\oplus d} \ar[u]^{s\circ b} & \zeta\oplus\zeta^{-1}(p+q-x) \ar[u]^{s\oplus r} \ar[r] & C_3 \ar[r] \ar[u]^{g}  & 0 }
\end{equation}
Where $h$ is defined by $h(u,v):=u-v$. Since $h \circ (s\oplus r)$ is surjective, $g$ is also surjective. Since degree of  $C_3$ is 2 and supported at $\{p,q\}$, $g$ is an isomorphism.  

In summary, $E$ fits into the diagram
\begin{equation}\label{eq6}
\xymatrix@=20pt{0 \ar[r] & E \ar[r]^(0.25){b\oplus d} & \zeta\oplus\zeta^{-1}(p+q-x) \ar[r]^(0.65){(v_p \oplus v_q)} & \CC_{p+q} \ar[r] & 0,} \end{equation}
such that  $v_t \in \CC^*$ for $t\in \{p,q\}$. We remark that $v_p$ and $v_q$ lie in $\CC^*=\PP^1\setminus \{[1:0], [0:1]\}$ by the surjectivities of $b$ and $d$.
Conversely, it is easy to check that a vector bundle $E$ satisfying these conditions is contained in $\PP V_{\zeta}^s\cap \PP V_{\eta}$ when $\eta \cong \zeta^{-1}(p+q-x)$. 
Therefore, by Proposition \ref{line}, we have the conclusion.
\end{proof}

To study the stable maps in $\bP$, we need the following simple conclusion in Corollary \ref{degree}.
\begin{lemm}\label{ninj}
For $\alpha \neq \beta \in \PP V_L^s$, $\Psi_L(\alpha)=\Psi_L(\beta)$ if and only if $\alpha$ and $\beta$ lie in a trisecant line.
\end{lemm}
\begin{proof}
Let $\Psi_L(\alpha)=\Psi_L(\beta)$, then
\[
\xymatrix{ \alpha : [ \ 0 \ar[r] & L^{-1}(-x) \ar[r]^(0.6)a & E \ar@{=}[d] \ar[r]^b & L \ar[r] & 0 \ ]
\\ \beta : [ \ 0 \ar[r] & L^{-1}(-x)\ar[r]^(0.6)c & E \ar[r]^d & L \ar[r] & 0 \ ].}
\]
If $d \circ a=0$, then $a$ factors through $L^{-1}(-x)$, and $a$ is clearly an isomorphism on $L^{-1}(-x)$. Also $d$ descends on $L$ and thus $d$ is an isomorphism on $L$. So we conclude that $\alpha=\beta$ which leads to a contradiction. Thus $d \circ a\neq 0$.  Hence $d \circ a$ is injective. This implies that $\Coker(d \circ a)= \CC_{p+q+r}$ for some $p,q,r\in X$. Then $L^{-1}(-x)\cong L(-p-q-r)$. That is, $L^2(x)\cong \cO_X(p+q+r)$. In a similar way, $b \circ c \neq 0$.
By Corollary \ref{trisec}, there is a trisecant line $l$ in $\PP^{g+1}$ such that $l\cap X=p+q+r$.
Consider a map $b\oplus d : E \to L\oplus L$. By the similar argument in equation \eqref{comdg2}, we have $\Coker(b\oplus d) = \CC_{p+q+r}$. Consider a following commutative diagram :
\[
\xymatrix{ 0 \ar[r] & L^{-1}(-x) \ar@{_{(}->}[d]^{a} \ar[r]^{d \circ a} & L \ar@{_{(}->}[d]^{0\oplus id} \ar[r] & \CC_{p+q+r} \ar@{=}[d] \ar[r] & 0
\\ 0 \ar[r] & E \ar[r]^{b \oplus d} & L\oplus L \ar[r] & \CC_{p+q+r} \ar[r] & 0 }
\]
By the snake lemma, we have a following diagram
\[
\xymatrix{ 0 \ar[r] & L^{-1}(-x) \ar@{_{(}->}[d]^{d\circ a} \ar[r]^{a} & E \ar@{_{(}->}[d]^{b\oplus d} \ar[r]^{b} & L \ar@{=}[d] \ar[r] & 0
\\ 0 \ar[r] & L \ar[r] & L\oplus L \ar[r] & L \ar[r] & 0. }
\]
Hence $\alpha$ and $\beta$ are contained in the projectivization of the kernel of the map $j_2$ in \eqref{seq2} because $L\cong L^{-1}(p+q+r-x)$. Hence by Corollary \ref{trisec}, $\alpha$ and $\beta$ lies in the trisecant line $l$. 

To prove the necessary condition, it is enough to show that the trisecant line contracts to a point in $\cN$. Let $l$ be an arbitrary trisecant line such that $l\cap X=p+q+r$. In this case, $l=\overline{pq}$.
Consider the following exact sequence.
\begin{align}\label{seq1}
0\lr \Ext^0(L^{-1}(p+q-x),L) \to \Ext^0(L^{-1}(p+q-x),L|_{p+q})\lr&\\
\Ext^1(L^{-1}(p+q-x),L(-p-q))\stackrel{j_3}{\longrightarrow} \Ext^1(L^{-1}(p+q-x),L) &\to  \Ext^1(L^{-1}(p+q-x),L|_{p+q})=0\nonumber,
\end{align}
be the long exact sequence from the short exact sequence $\ses{L(-p-q)}{L}{L|_{p+q}}$ by taking the functor $ \Hom(L^{-1}(p+q-x),-)$.
Let us regard elements of $\ker j_3$ as $\PP^{g-1}=\PP V_{L^{-1}(p+q-x)}$ represented by
\[ \ses{L(-p-q)}{E}{L^{-1}(p+q-x)}. \] 
By the same method as the proof of Proposition \ref{line}, $E$ fits into an exact sequence:
\begin{equation}\label{eq111}
\xymatrix@=20pt{0 \ar[r] & E \ar[r]^(0.2){b\oplus d} & L\oplus L^{-1}(p+q-x) \ar[r]^(0.65){v_p \oplus v_q} & \CC_{p+q} \ar[r] & 0}
\end{equation}
such that $d$ is surjective. So by Proposition \ref{line} it corresponds to a point $\overline{pq}\setminus{\{p,q\}}$. Conversely, by some diagram chasing, one can easily check that a vector bundle $E$ fits into the diagram \eqref{eq111} such that $d$ is surjective corresponds to an element of $\ker j_3$.
So $\overline{pq}\setminus{\{p,q\}}$ contracts to a point if and only if the dimension of $\ker j_3$ is $\leq1$.
By \eqref{seq1}, $$\ker j_3\cong \Ext^0(L^{-1}(p+q-x),L|_{p+q}) /\Ext^0(L^{-1}(p+q-x),L)\cong\CC^2/\rH^0(L^2(x)(-p-q)).$$
But $L^2(x)(-p-q)\cong \cO_X(r)$ (Corollary \ref{trisec}) and thus the claim holds.
\end{proof}

\begin{coro}\label{degree}
If the line $l$ meets with $X$ $n$ times (possibly with multiplicity), the map $i : l \setminus (l \cap X )\lr \cN$ is of degree $3-n$ for $n=0,1,2,3$.
\end{coro}
\begin{proof} 
$n=0$: This case is clear since $\deg\Psi_L=3$

$n=1$:  The degree of $\mathrm{deg}(i)\in \{0,1,2\}$. If $\mathrm{deg}(i)=1$, then it contracts to a point in $\cN$. Therefore by the Lemma \ref{ninj}, $l$ is a trisecant line of $X$ which leads to a contraction. If $\mathrm{deg}(i)=1$, by \cite{CH06}, the map $i$ factors through $\PP^{g-1}=\PP V_M$ for some $M\in \Pic^0(X)$. Thus by Corollary \ref{intersec}, we conclude that $l$ meets $X$ twice, which leads to the contraction. Therefore $\mathrm{deg}(i)=2$.

$n=2$: Let $l\cap X=p+q$ then $l=\overline{pq}$. Since $l$ is not trisecant, $\rH^0(L^2(x)(-p-q))=0$ by Proposition \ref{trisec}. Thus by the proof of Lemma \ref{ninj}, $l$ corresponds to a line in $\PP^{g-1}=\PP V_{L^{-1}(p+q-x)}$. Hence the map $\mathrm{deg}(i)=1$ since $\Psi_{L^{-1}(p+q-x)}$ is a linear embedding.

$n=3$:  Let $l\cap X=p+q+r$. By Lemma \ref{ninj}, $l$ contracts to a point in $\cN$. Therefore $\mathrm{deg}(i)=0$.

\end{proof}
\begin{rema}
Consider the $n=1$ case in Corollary \ref{degree}. Since $\mathrm{deg}(i)=0$, the closure $\overline{l}:=\overline{i(l \setminus (l \cap X ))}$ in $\cN$ is a smooth conic. By \cite[Proposition 3.6]{Kie07}, $\overline{l}$ should be a \emph{Hecke} curve or a conic in $\PP V_M\cong \PP^{g-1}$ for some $M\in \Pic^0(X)$. In the latter case, $l\cap X=r$, $l\setminus {r} \subset \PP V_L^s \cap \PP V_M$ for some $M\in \Pic^0(X)$. 
This contradicts part i) of Proposition \ref{intersec}
Hence $\overline{l}$ is a Hecke curve in $\cN$.
\end{rema}
From now on, let us fix a non-trisecant line bundle $L$ on $X$. So by Corollary $2.7$, n=3 case does not occur.
\section{Stable maps in the moduli space $\cN$}
In \S 1.1, we reviewed the classification of degree 3 stable maps $\PP^1 \to \cN$ studied in  \cite{Cas04, Kie07}. In this section, we study the closure of the component parametrizing the stable maps of type ii) in \S 1.1.
\subsection{Stable maps in the blown-up space}
Let $\bP=\text{bl}_X\PP^{g+1}$ be the blown-up space for the non-trisecant line bundle $L$. Let $\beta=\pi^*[\text{line}]$ be the curve class in $\rH_2(\bP)$. In this subsection, we will study the moduli space
\[
\bM_0(\bP,\beta)(\subset \bM_0(\cN,3))
\]
of stable maps of genus $0$ with embedded degree $3$. Let us start with the topological classification of the stable maps in $\bP$.
\begin{lemm}\label{clfy}
Stable maps parametrized by the closed points in $\bM_0(\bP,\beta)$ are one of the following types.
\begin{enumerate}
\item Lines in $\PP^{g+1}\setminus X$.
\item Gluing of the proper transform of a line in $\PP^{g+1}$ which meets $X$ at a point $p$ and a line in $\pi^{-1}(p) \cong \PP^{g-1}$.
\item Gluing of the proper transform of a line in $\PP^{g+1}$ which meets $X$ at two different points $p,q$, a line in $\pi^{-1}(p) \cong \PP^{g-1}$, and a line in $\pi^{-1}(q) \cong \PP^{g-1}$
\item Gluing of the proper transform of a line in $\PP^{g+1}$ which meets $X$ at two different points $p,q$ and a degree two stable map in $\pi^{-1}(p)\cong \PP^{g-1}$.
\item Gluing of the proper transform of a line in $\PP^{g+1}$ which is tangent to $X$ at $p$ and a degree two stable map in $\pi^{-1}(p)\cong \PP^{g-1}$.
\end{enumerate}
Furthermore, each dimension of the loci of theses types (1)-(5) is $2g$, $2g-1$, $2g-2$, $2g$ and $2g-1$ respectively.
\end{lemm}
\begin{proof}
Since $\rH_2(\bP)\cong \ZZ\oplus\ZZ$ where $(1,0)$ is a homology class of $\pi^*[\mathrm{line}]$ and $(0,1)$ is a homology class of a line in $\pi^{-1}(p)$, the homology class of the proper transform of a line $l\subset \PP^{g+1}$ which meets $X$ at $n$ point is $(1,-n)$. 
Because of the non-trisecant condition on the curve $X$ (Corollary \ref{trisec}), one can obtain the classifications of stable maps in $\bP$. The dimension computation is straightforward.
For example, let us compute the dimension of the locus consisting of type (4).  The locus of stable maps consisting of type (4) has a fibration over $X\times X\setminus \Delta$ with the fiber $Z$. Here $Z$ parameterizes the stable maps of degree two in $\PP^{g-1}$ passing through a fixed point. But one can easily check that $Z$ is irreducible by \cite{KP01} and \cite[Chapter III, Corollary 9.6]{Har77}.
Therefore the locus of stable maps consisting of the type (4) is of dimension $\mathrm{dim}Z+ 2=(2g-2)+2=2g$.
\end{proof}

Through the proof  of \cite[Corollary 4.6]{KL13}, one can regard the space $\bM_0(\bP,\beta)$ locally as the zero locus of the section of a vector bundle on a smooth space and thus all irreducible components of $\bM_0(\bP,\beta)$ have dimension at least $\int_{\beta=\pi^*[\mathrm{line}]}c_1(T_{\bP})+\mathrm{dim}\bP-3=2g$.
\begin{theo}\label{comp}
The moduli space $\bM_0(\bP,\beta)$
consists of two irreducible components $\overline{\Gamma}_1$ and $\overline{\Gamma}_2$, where:
\begin{enumerate}
\item $\Gamma_1$ parameterizes lines in $\PP^{g+1}\setminus X$.
\item $\Gamma_2$ parameterizes the union of a smooth conic and a line $l$ meeting at a point where $\pi(l)$ is a line that meets $X$ at two points (allowing $l$ to be tangential to $X$).
\end{enumerate}
Moreover, the intersection $\overline{\Gamma}_1\cap \overline{\Gamma}_2$ consists of stable maps of type (5) of Lemma \ref{clfy}.
\end{theo}
As we will see in the proof of this theorem, the loci of stable maps of types (1)-(3) and (5) (resp.  (4) and (5) ) are contained in $\overline{\Gamma}_1$ (resp. $\overline{\Gamma}_2$).
For the proof of Theorem \ref{comp}, let us start with the computation of the obstruction space of the stable maps of type (4).
\begin{lemm}\label{norline}
Let $l\subset \bP$ be the line such that $\pi(l)\cap X=\{p,q\}$, $p \neq q$. Then the normal bundle of the line $l$ is given by
\[
N_{l/\bP}\cong \cO_l(-1)^{\oplus (g-2)}\oplus \cO_l(-1)\oplus \cO_l(1) \mbox{ or } \cO_l(-1)^{\oplus (g-2)}\oplus \cO_l^{\oplus 2}.
\]
\end{lemm}
\begin{proof}
Let $l_0$ be the line in $\PP^{g+1}$ cleanly intersecting $X$ at two points $\{p,q\}$. Let $l$ be the strict transformation of the line $l_0$ along the blow-up map $\pi: \bP=\mbox{bl}_X \PP^{g+1}\lr \PP^{g+1}$.
By the proof of \cite[Lemma 1]{KLO07}, one can see that the bundle $N_{l/\bP}$ fits into the following short exact sequence:
\[
\ses{\pi^*N_{l_0/\PP^{g+1}}\otimes \cO(-E)|_l}{N_{l/\bP}}{\CC_p\oplus \CC_q}.
\]
In fact, the last map is locally given by the followings(cf. \cite[Appendix B.6.10]{Ful84}). 

Let $T_1,...,T_{g+1}$ be a local coordinate of $\PP^{g+1}$ at $p$ such that locally $I_{l_0/\PP^{g+1}}=\langle T_1,T_3,...,T_{g+1} \rangle$, $I_{X/\PP^{g+1}}=\langle T_2,T_3,...,T_{g+1} \rangle$. Then, we have local coordinate $t_1,t_2,x_3,...,x_{g+1}$ of $\bP$ at $\tilde{p}$ which is a lift of $p$ in $l$ such that $\pi\circ T_1 = t_1, \pi\circ T_2 = t_2, \pi\circ T_i = t_2 \circ x_i$ for $3\leq  i \leq g+1$. Therefore we have $\pi^*I_{l_0/\PP^{g+1}}=\pi^*\langle t_1,t_3,...,t_{g+1} \rangle = \langle t_1,t_2x_3,...,t_2x_{g+1} \rangle$. Locally, $\langle t_2 \rangle$ is an ideal of exceptional divisor $E$. Therefore we have a short exact sequence.
\[
0 \to I_{l/\bP}\cdot I_{E/\bP} \to \pi^*I_{l_0/\PP^{g+1}} \to \pi^*I_{l_0/\PP^{g+1}}/I_{l/\bP}\cdot I_{E/\bP} \to 0. 
\]
Take pull-back this sequence to $l$, we have
\[
0 \to I_{l/\bP}/I_{l/\bP}^2 \otimes \cO_{\bP}(-E)\to \pi^*(I_{l_0/\PP^{g+1}}/I_{l_0/\PP^{g+1}}^2) \stackrel{\tilde{\partial}_p}{\longrightarrow} \CC_p \to 0.
\]
Here, the map $\tilde{\partial}_p$ is a differentiation by a tangent vector $\frac{\partial}{\partial t_1}$ in $T_p\bP$. Dualizing this sequence we have a map $N_{l/\bP} \to \CC_p$. Similarly, a map $N_{l/\bP} \to \CC_q$ can be defined.

From $N_{l_0/\PP^{g+1}}=\cO_{l_0}(1)^{\oplus g}$ and $\cO(-E)|_l=\cO_l(-2)$, we obtain the result.
\end{proof}
Similar to other Fano varieties, the normal bundle of the line can be described by the following geometric way.
\begin{coro}
If the projectivized tangent line $T_pX$ and $T_qX$ are co-planar (resp. skew lines), then $N_{l/\bP}\cong \cO_l(-1)^{\oplus (g-2)}\oplus \cO_l(-1)\oplus \cO_l(1)$ (resp. $N_{l/\bP}\cong \cO_l(-1)^{\oplus (g-2)}\oplus \cO_l^{\oplus 2}$).
\end{coro}
\begin{proof}
One can see the results by using a local computation similar to that in the proof of Lemma \ref{norline}.
\end{proof}
Let $Q$ be the smooth conic in the exceptional divisor $E$. Since $Q\subset \PP^{g-1}$ for some fiber of the projective bundle $E=\PP(N_{X/\PP^{g+1}})\lr X$, we know that 
\[N_{Q/E}\cong N_{Q/\PP^{g-1}}\oplus N_{\PP^{g-1}/E}|_Q\cong (\cO_Q(2)\oplus\cO_Q(1)^{\oplus (g-3)})\oplus \cO_Q\] because $\rH^1(\cO_Q(i))=0$ for $i=1,2$. Therefore the nested bundle sequence
\begin{equation}\label{eq8}
\ses{N_{Q/E}}{N_{Q/\bP}}{N_{E/\bP}|_Q\cong \cO_Q(-1)}
\end{equation}
says that 
\begin{equation}\label{eq9}
N_{Q/\bP}\cong (\cO_Q(2)\oplus\cO_Q(1)^{\oplus (g-3)})\oplus \cO_Q\oplus \cO_Q(-1).
\end{equation}
\begin{prop}\label{obst}
Let $[C]\in \Gamma_2$ be the point such that $C=l\cup Q$ for a line $l$ and a smooth conic $Q$ meeting at a point $z$. Then $\rH^1(N_{C/\bP})=0$.
\end{prop}
\begin{proof}
Since the curve $C$ has only a nodal singularity, the conormal sheaf $N_{C/\bP}^*:=I_{C/\bP}/I_{C/\bP}^2$ is locally free.
From the two short exact sequences:
\begin{itemize}
\item $\ses{N_{C/\bP}^*}{\Omega_{\bP}|_C}{\Omega_{C}}$ and
\item $\ses{\cO_l(-1)}{\cO_C}{\cO_Q}$,
\end{itemize}
we have the following commutative diagram:
\[
\xymatrix{\Ext^1(\Omega_C, \cO_C)\ar[r]\ar@{->>}[d]&\Ext^1(\Omega_{\bP}, \cO_C)\ar[r]\ar[d]^{\cong}&\Ext^1(N_{C/\bP}^*,\cO_C)\ar[r]\ar[d]&0\\
\Ext^1(\Omega_C, \cO_Q)\ar[r]&\Ext^1(\Omega_{\bP}, \cO_Q)\ar[r]&\Ext^1(N_{C/\bP}^*,\cO_Q)\ar[r]&0.}
\]
Since $C=l\cup Q$ has the unique node point $z$, $\Ext^1(\Omega_C, \cO_C)\cong \CC$. Also, $\Ext^2(\Omega_C, \cO_L(-1))=0$ implies that the first vertical map is surjective.
By Lemma \ref{obst2}, the second vertical map $\rH^1(T_{\bP}|_C)\cong\Ext^1(\Omega_{\bP}, \cO_C)\lr \Ext^1(\Omega_{\bP}, \cO_Q)\cong \rH^1(T_{\bP}|_Q)=\CC$ is an isomorphism. Hence the claim holds whenever $\rH^1(N_{C/\bP}|_Q)=0$.
Let \[0\lr N_{C/\bP}^*|_Q\lr N_{Q/\bP}^*\stackrel{\partial_z}{\lr} \CC_z\lr0\] be the structure sequence where the map $\partial_z$ is defined by a differentiation by the tangent vector $T_zl$.
This can be shown by the following local computation. We can take local coordinate $x_1,...,x_{g+1}$ of $\bP$ at $z$. where $I_{Q/\bP}=\langle x_2,x_3,...,x_{g+1} \rangle, I_{l/\bP}= \langle x_1,x_3,...,x_{g+1} \rangle$. Then, $I_{C/\bP}= \langle x_1x_2,x_3,...,x_{g+1} \rangle$. Consider the exact sequene
\[ 
0 \to I_{C/\bP} \to I_{Q/\bP} \to I_{Q/\bP}/I_{C/\bP} \to 0
\]
Taking pull-back to $Q$, we have
\[
0 \to I_{C/\bP}/I_{C/\bP}|_Q \to I_{Q/\bP}/I_{Q/\bP}^2 \to \CC_z \to 0
\]
Here, $I_{Q/\bP}^2 \to \CC_z$ is given by the differentiation by $T_pl$ since it annihilates the local coordinates $x_1,x_3,...,x_{g+1}$.
Then one can check that the composition map $\cO_Q(1)\cong N_{E/\bP}^*|_Q\subset N_{Q/\bP}^*\stackrel{r}{\lr} \CC_p$ by \eqref{eq8} is not zero because $l$  transversally meets with the exceptional divisor $E$. Hence one can easily see that $N_{C/\bP}^*|_Q\cong \cO_Q(s) \oplus N^*_{Q/E}$ for some $s\in Q$. Since $\rH^1(\cO_Q(-r))=\rH^1(N_{Q/E})=0$, we proved the claim.
\end{proof}
\begin{lemm}\label{obst2}(cf. \cite[Lemma 6.4]{Kie07})
Let
\[
H^0(T{\bP}|_l)\oplus H^0(T{\bP}|_Q) \stackrel{\alpha}{\lr}\rH^0(T_p{\bP}) \lr H^1(T{\bP}|_C) \lr H^1(T{\bP}|_l)\oplus H^1(T{\bP}|_Q)\lr 0.
\]
be the long exact sequence coming from $\ses{\cO_C}{\cO_l\oplus \cO_Q}{\CC_p}$. Then the map $\alpha$ is surjective and thus $\rH^1(T{\bP}|_C)\cong H^1(T{\bP}|_Q)\cong \CC$.
\end{lemm}
\begin{proof}
$T_p\bP\cong T_pl\oplus T_pE$ since $l$ and $E$ transverally meets at the point $p$. From $\rH^0(T\bP|_l)=\rH^0(Tl)\oplus \rH^0(N_{l/\bP})$, we see that 
\begin{equation}\label{eq11}
\rH^0(T\bP|_l)\twoheadrightarrow \rH^0(T_pl)\end{equation} is surjective by the positive degree part $\rH^0(Tl)=\rH^0(\cO_l(2))$. 

On the other hand, $\rH^0(T\bP|_Q)=\rH^0(T\PP^{g-1}|_Q)\oplus \rH^0(N_{\PP^{g-1}/\bP}|_Q)$ from $Q\subset \PP^{g-1}\subset E$. One can easily see that the first factor $$\rH^0(T\PP^{g-1}|_Q)\twoheadrightarrow \rH^0(T_p\PP^{g-1})$$ is surjective. Also, by a simple computation, one can see that $N_{\PP^{g-1}/\bP}|_Q\cong \cO_Q\oplus \cO_Q(-1)$ where $N_{\PP^{g-1}/E}|_Q=\cO_Q$ and thus the positive degree part $\rH^0(N_{\PP^{g-1}/\bP}|_Q)=\rH^0(N_{\PP^{g-1}/E}|_Q)$ goes to $\rH^0(N_{\PP^{g-1}/E,p})=\CC$. Hence 
\begin{equation}\label{eq12}
\rH^0(T{\bP}|_Q)=\rH^0(T\PP^{g-1}|_Q)\oplus \rH^0(N_{\PP^{g-1}/\bP}|_Q)\twoheadrightarrow \rH^0(T_p\PP^{g-1})\oplus \rH^0(N_{\PP^{g-1}/E,p})=\rH^0(T_pE)\end{equation} is surjective. By \eqref{eq11} and \eqref{eq12}, the map $\alpha$ is surjective.
The last isomorphisms comes from  Lemma \ref{norline} and the equation \eqref{eq9}.
\end{proof}
Now we are ready to prove our main theorem.
\begin{proof}[Proof of Theorem \ref{comp}]
$\Gamma_1$  is isomorphic to an open subset of $Gr(2,g+2)$ which corresponds to a line in $\PP^{g+1}$ that does not meet $X$. So $\overline{\Gamma}_1$ is irreducible. Also $\Gamma_2$ is irreducible by the proof of Lemma \ref{clfy}.
By Lemma \ref{clfy},  both $\overline{\Gamma}_1$ and $\overline{\Gamma}_2$ have dimension $2g$, which agrees with the expected dimension. Moreover there is no other component which has dimension $\geq 2g$ (\cite[Proof of corollary 4.6]{KL13}). So $\overline{\Gamma}_1$ and $\overline{\Gamma}_2$ are all the irreducible components of $\bM_0(\bP,\beta)$. The loci of stable maps of type (1), (2) and (3) are not contained in $\overline{\Gamma}_2$, so it should be contained in  $\overline{\Gamma}_1$. The loci of stable maps of type (4) and (5) are contained in $\overline{\Gamma}_2$ by definition.

Since all irreducible components of $\bM_0(\bP,\beta)$ have dimension $2g$, it is a locally complete intersection by the proof of \cite[Corollary 4.6]{KL13}. A stable map of type (3) degenerates to a gluing of proper transform of a line in $\PP^{g+1}$ which is tangent to $X$ and a degenerate conic in $\pi^{-1}(p)\cong\PP^{g-1}$ when $p \to q$, which is a stable map of type (4). Hence $\overline{\Gamma}_1\cap\overline{\Gamma}_2 \neq \phi$. Also since there are only two irreducible components $\overline{\Gamma}_1$ and $\overline{\Gamma}_2$, $\overline{\Gamma}_1\cap\overline{\Gamma}_2$ is pure of dimension $2g-1$ by Hartshorne's connectedness theorem(\cite[Theorem 3.4]{Har62}).

By the proof of Proposition $\ref{obst}$, a stable map of type (4) whose conic component is smooth has no obstruction. So it is a smooth point in the moduli. So it cannot be contained in $\overline{\Gamma}_1\cap 
\overline{\Gamma}_2$. The sublocus of type (4) whose conic component is singular has dimension $2g-2$. Since the locus of stable maps of type (5) is dimension $2g-1$ and obviously irreducible, we conclude that $\overline{\Gamma}_1\cap \overline{\Gamma}_2$ parametrizes stable maps of type (5).
\end{proof}

\bibliographystyle{amsplain}

\end{document}